\documentclass[12pt]{amsart}
\usepackage{amsmath}
\usepackage{amssymb}
\usepackage{amsfonts}
\usepackage{euscript}

\newcommand{\Z}{\mathbb{Z}} \newcommand{\C}{\mathbb{C}}
\newcommand{\R}{\mathbb{R}} \newcommand{\Q}{\mathbb{Q}}
 
\newcommand{\A}{\mathbb{A}}
\newcommand{\fin}{{\rm f}}
\newcommand{\EE}{\EuScript {E}}
\newcommand{\OO}{\EuScript {O}}

\newcommand{\rd}{{\operatorname{rd}}}
\newcommand{\ird}{{\operatorname{ird}}}
\newcommand{\res}{\operatorname{Res}}

\newcommand{\cs}{\mathcal S}

\newcommand{\esA}{{\EuScript{A}}}
\newcommand{\esB}{{\EuScript{B}}}

\newtheorem{thm}{Theorem}[section]
\newtheorem{lem}[thm]{Lemma}
\newtheorem{prop}[thm]{Proposition}
\newtheorem{cor}[thm]{Corollary}

\theoremstyle{definition}
\newtheorem{defn}[thm]{Definition}
\newtheorem{rem}[thm]{Remark}
\numberwithin{equation}{section}

\begin{document}

\title[binary cubic forms]
{On relations among Dirichlet series\\ whose coefficients are\\
class numbers of binary cubic forms}
\date{\today}

\begin{abstract}
We study the class numbers of integral binary cubic forms.
For each $\mathrm{SL}_2(\Z)$ invariant lattice $L$,
Shintani introduced Dirichlet series whose coefficients
are the class numbers of binary cubic forms in $L$.
We classify the invariant lattices, and investigate
explicit relationships between
Dirichlet series associated with those lattices.
We also study the analytic properties of the Dirichlet series,
and rewrite the functional equation in a self dual form
using the explicit relationship.
\end{abstract}

\author[Yasuo Ohno]{Yasuo Ohno}
\address{(Y. Ohno)
Department of Mathematics, Kinki University,
Kowakae 3-4-1, Higashi-Osaka, Osaka 577-8502, Japan/
Max-Planck-Institut f\"{u}r Mathematik,
Vivatsgasse 7, 53111 Bonn, Germany}
\email{ohno@math.kindai.ac.jp}
\author[Takashi Taniguchi]{Takashi Taniguchi}
\address{(T. Taniguchi)
Department of Mathematical Sciences,
Faculty of Science,
Ehime University,
Bunkyocho 2-5, Matsuyama-shi,
Ehime, 790-8577,
Japan}
\email{tani@math.sci.ehime-u.ac.jp}
\author[Satoshi Wakatsuki]{Satoshi Wakatsuki}
\address{(S. Wakatsuki)
Department of Mathematics, Graduate School of Science,
Kanazawa University, Kakumamachi, Kanazawa, Ishikawa, 920-1192, Japan}
\thanks{
The first author is supported by JSPS Grant-in-Aid No.\!\! 18740020.
The second author is supported by
Research Fellowships for Young Scientists
of JSPS. The third author is supported by JSPS Grant-in-Aid No.\!\! 18840018}
\email{wakatuki@kenroku.kanazawa-u.ac.jp}

\maketitle

\section{Introduction}\label{sec:intro}

Study of the class numbers of integral binary cubic forms
was initiated by G. Eisenstein and developed by many mathematicians
including C. Hermite, F. Arndt, H. Davenport and T. Shintani.
Davenport \cite{davenport} obtained asymptotic formulas
for the sum of the class numbers of
integral irreducible binary cubic forms
of positive and negative discriminants.
Shintani \cite{shintanib} improved the error term
by using the Dirichlet series whose coefficients
are the class numbers of binary cubic forms introduced in \cite{shintania}.

Let $V_\Q$ be the space of binary cubic forms over
the rational number field $\Q$;
\[
V_\Q=\{x(u,v)=x_1u^3+x_2u^2v+x_3uv^2+x_4v^3\mid x_1,\dots,x_4\in\Q\}.
\]
For $x\in V_\Q$, the discriminant $P(x)$ is defined by
$P(x)=x_2^2x_3^2+18x_1x_2x_3x_4-4x_1x_3^3-4x_2^3x_4-27x_1^2x_4^2$.
The group $\Gamma=\mathrm{SL}_2(\Z)$ acts on $V_\Q$ by the linear
change of variables and $P(x)$ is invariant under the action.
Let $L$ be a $\Gamma$-invariant lattice in $V_\Q$.
We put $L_\pm=\{x\in L\mid \pm P(x)>0\}$.
For $x\in L$, let $\Gamma_x$ be the
stabilizer of $x$ in $\Gamma$ and ${}^{\#}\Gamma_x$ its order.
\begin{defn}
For each invariant lattice $L$ and sign $\pm$, we put
\[
\tilde\xi_{\pm}(L,s)
	:=\sum_{x\in\Gamma\backslash L_\pm}
	\frac{({}^{\#}\Gamma_x)^{-1}}{|P(x)|^s}.
\]
\end{defn}
This Dirichlet series was introduced by
Shintani \cite{shintania} as an example of the zeta functions
of prehomogeneous vector spaces. It is shown
that this Dirichlet series has number of curious properties
such as analytic continuation or functional equation.
He treated when the invariant lattice is either
$L_1=\{x\in V_\Q\mid x_1,x_2,x_3,x_4\in\Z\}$ or
$L_2=\{x\in V_\Q\mid x_1,x_4\in\Z, x_2,x_3\in3\Z\}$,
but the proof works for a general invariant lattice
as we confirm in this paper.
Note that $L_1$ and $L_2$ are the dual lattice to each other
with respect to the alternating form
$\langle x,y\rangle=x_1y_4-3^{-1}x_2y_3+3^{-1}x_3y_2-x_4y_1$
on $V_\Q$.

In 1997, the first author \cite{ohno} conjectured
that there are simple relations between $\tilde\xi_\mp(L_1,s)$
and $\tilde\xi_\pm(L_2,s)$. This was proved by Nakagawa \cite{nakagawa}.
\begin{thm}[Conjectured in \cite{ohno}, proved in \cite{nakagawa}]\label{thm:intronakagawa}
\[
\tilde\xi_-(L_1,s)=3^{3s}\tilde\xi_+(L_2,s)
\qquad\text{and}\qquad
\tilde\xi_+(L_1,s)=3^{3s-1}\tilde\xi_-(L_2,s).
\]
\end{thm}

The primary purpose of this paper
is to classify the $\Gamma$-invariant lattices
and investigate whether
there are similar formulas for those lattices.
In Section \ref{sec:extrafneq} we prove the following.

\begin{thm}[Theorem \ref{thm:extrafneqL7L10}]\label{thm:intro}
There are $10$ kinds of $\Gamma$-invariant lattices up to scaling.
If we denote these lattices by $L_1,\dots,L_{10}$
as in Theorem \ref{thm:classifyinvlat}, then for
Dirichlet series associated with $L_7,\dots,L_{10}$ we have
\begin{align*}
\tilde\xi_-(L_7,s)&=3^{3s}\tilde\xi_+(L_8,s),
&
\tilde\xi_+(L_7,s)&=3^{3s-1}\tilde\xi_-(L_8,s),\\
\tilde\xi_-(L_9,s)&=3^{3s}\tilde\xi_+(L_{10},s),
&
\tilde\xi_+(L_9,s)&=3^{3s-1}\tilde\xi_-(L_{10},s).
\end{align*}
On the other hand, the
Dirichlet series associated with $L_3,\dots,L_6$
do not satisfy such simple relations as above.
For example, $\tilde\xi_-(L_3,s)$ and $3^{3s}\tilde\xi_+(L_4,s)$
do not coincide with each other.
\end{thm}

These relations of the Dirichlet series
are proved in Theorem \ref{thm:extrafneqL7L10} using
Theorem \ref{thm:intronakagawa}.
(In Section \ref{sec:extrafneq}
we slightly modify the definition of the Dirichlet series.)
It is likely that the relations
among the Dirichlet series for $L_3,\dots,L_6$
are somewhat more complicated.
If we take the arithmetic subgroup
$\Gamma$ smaller, there appears more invariant lattices
and it may be an interesting problem to study Dirichlet series
associated with those lattices.
We hope these problems to be answered in the future.

Such a relation of the Dirichlet
is expected to exist also for some other representations.
Among them for the space of pairs of ternary quadratic forms
$(G,V)=({\rm GL}_3\times{\rm GL}_2,
({\rm Sym^2Aff^3})^\ast\otimes {\rm Aff^2})$,
this problem is considerably interesting and being studied
by several mathematicians including Bhargava and Nakagawa.
We note that there are only $2$ types
of $G_\Z$-invariant lattices for this case.

We explain a curious application of
this theorem to the functional equation
for $\tilde\xi_{\pm}(L_i,s)$.
Let $a_1=a_2=0$ and $a_3=\dots=a_{10}=2$.
Following Datskovsky and Wright \cite{dawra}
we put
\begin{align*}
\Lambda_\pm(L_i,s)
&	:=\frac{2^{(a_i+1)s}3^{3s/2}}{\pi^{2s}}
		\Gamma(s)\Gamma(\frac{s}{2}+\frac14\mp\frac13)
		\Gamma(\frac{s}{2}+\frac14\mp\frac16)
	\left(\sqrt3\tilde\xi_+(L_i,s)\pm\tilde\xi_-(L_i,s)\right)
\end{align*}
for each sign.
Then Shintani's functional equation
between the vector valued functions
$(\tilde\xi_+(L_i,1-s),\tilde\xi_-(L_i,1-s))$
and $(\tilde\xi_+(L_{i+1},s),\tilde\xi_-(L_{i+1},s))$
$(i=1,3,5,7,9)$ 
is diagonalized and symmetrized as
\begin{align*}
\Lambda_\pm(L_{i},1-s)&=\pm2^{a_i-b_i}3^{3s-1/2}\Lambda_\pm(L_{i+1},s)
\end{align*}
where $b_1=0$, $b_3=1$, $b_5=3$ and $b_7=b_9=2$.
Let $i$ be either $1$, $7$ or $9$.
Then Theorems \ref{thm:intronakagawa} and \ref{thm:intro} state that
$\Lambda_\pm(L_{i+1},s)=\pm3^{1/2-3s}\Lambda_\pm(L_{i},s)$.
Since $a_i=b_i$ holds also,
we can write the functional equations above as follows.
\begin{thm}[Theorem \ref{cor:singlefneq}]
Let $i$ be either $1$, $7$ or $9$. Then
\begin{equation*}
\Lambda_\pm(L_{i},1-s)=\Lambda_\pm(L_{i},s).
\end{equation*}
A similar formula holds for $i=2$, $8$ or $10$.
\end{thm}
The case $i=1,2$ is stated in \cite[p.1088]{ohno}.
Unlike Shintani's original one,
this functional equation is of the single Dirichlet series
$\sqrt3\tilde\xi_+(L_i,s)\pm\tilde\xi_-(L_i,s)$
and also the equation is completely symmetric.
We hope this equation
might help us to know something on the real nature
of the Dirichlet series.
Note that the Dirichlet series
$\sqrt3\tilde\xi_+(L_i,s)\pm\tilde\xi_-(L_i,s)$
does not have an Euler product for any $L_i$
(see Proposition \ref{prop:EulerProduct}.)
\bigskip

This paper is organized as follows.
In Section \ref{sec:lattices}, we give the classification
of the invariant lattices without a proof.
The proof is given in Section \ref{sec:prooflattice}.
In Section \ref{sec:extrafneq}, we study
the explicit relationship of the Dirichlet series.
In Section \ref{sec:residue}
we study the analytic properties of the Dirichlet series.
In Theorem \ref{thm:analyticprop} we give
functional equations explicitly and evaluate the residues of the poles.
After that we study on the diagonalization of the functional equation
and give a simple symmetric functional equation using the result of Section \ref{sec:extrafneq}.
We also give in Theorem \ref{thm:density}
the density of the class numbers of the lattices.
In Section \ref{sec:coefftable}, we give a table
of about first fifty coefficients of the Dirichlet series.
\bigskip

\noindent
{\bf Acknowledgments.}
Dr.\!\! Noriyuki Abe wrote a good C++ program to compute
the coefficients of the Dirichlet series.
The table of coefficients played an important role
in studying the Dirichlet series.
The authors express their deep gratitude to him.
The authors are also grateful to Professor Tomoyoshi Ibukiyama
for useful comments, especially on applications of our results
to the functional equations.

\bigskip

\noindent
{\bf Notations.}
The standard symbols $\Q$, $\R$, $\C$ and $\Z$ will denote respectively
the set of rational, real and complex numbers and the rational integers.
If $V$ is a variety defined over a ring $R$ and $S$ is an $R$-algebra
then $V_S$ denotes its $S$-rational points.
The $1$-dimensional affine space is denoted by $\mathrm{Aff}$.

\section{Classification of invariant lattices}
\label{sec:lattices}
Let $G$ be the general linear group of rank $2$
and $V$ the space of binary cubic forms;
\begin{align*}
G&=\mathrm{GL}_2,\\
V&=
\{x=x(v_1,v_2)=x_1v_1^3+x_2v_1^2v_2+x_3v_1v_2^2+x_4v_2^3\mid x_i\in\mathrm{Aff}\}.
\end{align*}
We identify $V$ with $\mathrm{Aff}^4$ via the map
$x\mapsto (x_1,x_2,x_3,x_4)$.
We define the action of $G$ on $V$ by
\[
(gx)(v_1,v_2)=\frac{1}{\det(g)}\cdot x(pv_1+rv_2,qv_1+sv_2),
\qquad
g=\begin{pmatrix}p&q\\r&s\end{pmatrix}\in G,
\ \ 
x\in V.
\]
The twist by $\det(g)^{-1}$ is to make the representation faithful.
For $x\in V$, let $P(x)$ be the discriminant;
\begin{equation*}
P(x)=x_2^2x_3^2-4x_1x_3^3-4x_2^3x_4+18x_1x_2x_3x_4-27x_1^2x_4^2.
\end{equation*}
Then we have $P(gx)=(\det g)^2P(x)$.
We put $G^1=\mathrm{SL}_2$.
We assume these are defined over $\Z$.

Let $\Gamma\subset G_\Q$ be an arithmetic subgroup.
The zeta functions of the prehomogeneous vector space $(G,V)$ over $\Q$
are defined for each $\Gamma$-invariant lattice in $V_\Q$.
In this paper we consider the case $\Gamma=G^1_\Z=\mathrm{SL}_2(\Z)$.
To begin we need the classification of the invariant lattices.
For a lattice $L$ in $V_\Q$ and $q\in\Q^\times$,
we put $qL=\{qx\mid x\in L\}$. Then if $L$ is a $\Gamma$-invariant
lattice, $qL$ is $\Gamma$-invariant also.
Up to such a scaling,
$G^1_\Z$-invariant lattices are classified as follows.

\begin{thm}\label{thm:classifyinvlat}
Up to scaling, the following is a complete list 
of $\mathrm{SL}_2(\Z)$-invariant lattices in $V_\Q$:
\begin{eqnarray*}
L_1&=&\{(a,b,c,d)\in \Z^4 \}\\
L_2&=&\{(a,3b,3c,d)\in \Z^4 \ |\  b,c\in \Z \}\\
L_3&=&\{(a,b,c,d)\in L_1 \ |\  b+c\in 2\Z \}\\
L_4&=&\{(a,3b,3c,d)\in L_2 \ |\  a,d,b+c\in 2\Z \}\\
L_5&=&\{(a,b,c,d)\in L_1 \ |\  a,d,b+c\in 2\Z \}\\
L_6&=&\{(a,3b,3c,d)\in L_2 \ |\  b+c\in 2\Z \}\\
L_7&=&\{(a,b,c,d)\in L_1 \ |\  a+b+c,b+c+d \in 2\Z \}\\
L_8&=&\{(a,3b,3c,d)\in L_2 \ |\  a+b+d,a+c+d \in 2\Z \}\\
L_9&=&\{(a,b,c,d)\in L_1 \ |\  a+b+d,a+c+d \in 2\Z \}\\
L_{10}&=&\{(a,3b,3c,d)\in L_2 \ |\  a+b+c,b+c+d \in 2\Z \}
\end{eqnarray*}
\end{thm}

We give a proof of this theorem in Section \ref{sec:prooflattice}.
Each of $L_3,L_5,L_7,L_9$ is a sublattice of $L_1$ and is containing $2L_1$.
The relations of inclusions and their indices are given by
\begin{gather*}
[L_1:L_3]=[L_3:L_9]=[L_7:L_5]=[L_5:2L_1]=2,\\
[L_1:L_7]=[L_3:L_5]=[L_9:2L_1]=4.
\end{gather*}
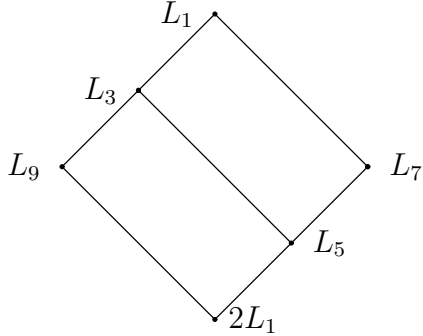
\begin{center}
\unitlength 0.1in
\begin{picture}( 20.2500, 16.8500)(  3.7500,-22.0000)
%
\special{pn 8}%
\special{pa 1600 600}%
\special{pa 800 1400}%
\special{pa 1600 2200}%
\special{pa 2400 1400}%
\special{pa 1600 600}%
\special{pa 1600 600}%
\special{pa 1600 600}%
\special{fp}%
%
\special{pn 8}%
\special{pa 2000 1800}%
\special{pa 1200 1000}%
\special{pa 1200 1000}%
\special{pa 1200 1000}%
\special{fp}%
\put(14.0000,-6.0000){\makebox(0,0){$L_1$}}%
%
\special{pn 8}%
\special{sh 1}%
\special{ar 1600 600 10 10 0  6.28318530717959E+0000}%
\special{sh 1}%
\special{ar 1200 1000 10 10 0  6.28318530717959E+0000}%
\special{sh 1}%
\special{ar 1200 1000 10 10 0  6.28318530717959E+0000}%
%
\special{pn 8}%
\special{sh 1}%
\special{ar 800 1400 10 10 0  6.28318530717959E+0000}%
\special{sh 1}%
\special{ar 1600 2200 10 10 0  6.28318530717959E+0000}%
\special{sh 1}%
\special{ar 2000 1800 10 10 0  6.28318530717959E+0000}%
\special{sh 1}%
\special{ar 2400 1400 10 10 0  6.28318530717959E+0000}%
\special{sh 1}%
\special{ar 2400 1400 10 10 0  6.28318530717959E+0000}%
\put(10.0000,-10.0000){\makebox(0,0){$L_3$}}%
\put(6.0000,-14.0000){\makebox(0,0){$L_9$}}%
\put(26.0000,-14.0000){\makebox(0,0){$L_7$}}%
\put(22.0000,-18.0000){\makebox(0,0){$L_5$}}%
\put(18.0000,-22.0000){\makebox(0,0){$2L_1$}}%
\end{picture}%
\end{center}
There are similar relations for $L_2,\dots,L_{10}$.

We define the alternating form on $V_\Q$ by
$\langle x,y\rangle=x_1y_4-3^{-1}x_2y_3+3^{-1}x_3y_2-x_4y_1$.
Then $L_i$ and $2^{-1}L_{i+1}$ are the dual lattices to each other
for $i=3,5,7,9$.

\begin{rem}
We immediately see that
all of the lattices in Theorem \ref{thm:classifyinvlat}
are invariant under the action of
$(\begin{smallmatrix}0&1\\1&0\\\end{smallmatrix})\in G_\Z$.
Since the group $G_\Z=\mathrm{GL}_2(\Z)$ is generated by
$(\begin{smallmatrix}0&1\\1&0\\\end{smallmatrix})$
and $G^1_\Z$,
Theorem \ref{thm:classifyinvlat}
also gives the list of $\mathrm{GL}_2(\Z)$-invariant lattices.
\end{rem}

\section{Relations of the Dirichlet series}
\label{sec:extrafneq}
In this section, we define the Dirichlet series
for each lattice and study their relations.
Let $L_i^+=\{x\in L_i\mid P(x)>0\}$ and $L_i^-=\{x\in L_i\mid P(x)<0\}$.
For $x\in L_i$, we put
$G^1_{\Z,x}=\{\gamma\in{\rm SL}_2(\Z)\mid \gamma x=x\}$
and denote by ${}^{\#} G^1_{\Z,x}$ its order.
We note that ${}^{\#} G^1_{\Z,x}$ is either $1$ or $3$.
\begin{defn}\label{defn:Dirichlet}
\begin{enumerate}
\item
For $i=1,3,5,7,9$, we put
\[
\xi_\pm(L_i,s)=\sum_{x\in{G^1_\Z}\backslash L_i^\pm}
\frac{({}^{\#} G^1_{\Z,x})^{-1}}{|P(x)|^s}.
\]
\item
For $i=2,4,6,8,10$, we put
\[
\xi_\pm(L_i,s)=3^{3s}\sum_{x\in{G^1_\Z}\backslash L_i^\pm}
\frac{({}^{\#} G^1_{\Z,x})^{-1}}{|P(x)|^s}.
\]
\end{enumerate}
\end{defn}
These Dirichlet series were introduced by Shintani \cite{shintania}
as an example of the zeta functions of prehomogeneous vector spaces.
This definition in (2) differs from that in \cite{shintania}
by the factor of $3^{3s}$.
Note that if $x\in L_2$ then $P(x)$ is a multiple of $3^3$.
It is known that these Dirichlet series converges for $\Re(s)>1$.
The analytic properties are studied in Section \ref{sec:residue}.

In \cite{ohno}, the first author gave the following conjecture,
and proved that if the conjecture is true then
the Shintani's functional equation
has a simple symmetric form.
This conjecture was proved by Nakagawa \cite{nakagawa}.
\begin{thm}[Nakagawa]\label{thm:extrafneqL1L2}
\[
\xi_-(L_1,s)=\xi_+(L_2,s),
\qquad
3\xi_+(L_1,s)=\xi_-(L_2,s).
\]
\end{thm}
In this section, we prove the following analogous relations.
The simplification and symmetrization of Shintani's functional equation
in terms of this theorem is given in Theorem \ref{cor:singlefneq}.
\begin{thm}\label{thm:extrafneqL7L10}
\begin{align*}
\xi_-(L_7,s)&=\xi_+(L_8,s),\\
\xi_-(L_9,s)&=\xi_+(L_{10},s),\\
3\xi_+(L_7,s)&=\xi_-(L_8,s),\\
3\xi_+(L_9,s)&=\xi_-(L_{10},s).
\end{align*}
\end{thm}
On the other side
the table in Section \ref{sec:coefftable} asserts
that, for example, $\xi_-(L_3,s)$ and $\xi_+(L_4,s)$
do not coincide with each other.
We will reduce Theorem \ref{thm:extrafneqL7L10}
to Theorem \ref{thm:extrafneqL1L2}.
The proof is given after Proposition \ref{pr:L8L10}.

To prove this theorem, we study the relation between different lattices.
Let $\EE$ and $\OO$ be the set of even integers and odd integers,
respectively;
\[
\EE
=\{2n\mid n\in\Z\},
\quad
\OO
=\{2n+1\mid n\in\Z\}.
\]
We write elements of $L_1=\Z^4$ as $x=(a,b,c,d)$
in this section.
Hence 
\[
P(x)=b^2c^2+18abcd-4ac^3-4b^3d-27a^2d^2.
\]
We first consider the lattices in $L_1$.
We put
$\Delta=ac^3+b^3d-a^2d^2$.
Then
\[
P(x)=(bc+ad)^2-4\Delta+16(abcd-2a^2d^2).
\]

\begin{defn}
Let $L$ be a lattice in $L_1$.
For $l,N\in\Z$, $N\neq0$, we put
\[
L_{\equiv l\,(N)}=\{x\in L\mid P(x)\equiv l\mod N\}.
\]
\end{defn}

\begin{prop}\label{pr:L7L9}
We have
\begin{align*}
L_7&=2L_1\amalg L_{1,\equiv 1\,(8)},\\
L_9&=2L_1\amalg L_{1,\equiv 5\,(8)},
\end{align*}
\end{prop}
We start with a lemma.
\begin{lem}
Let $x=(a,b,c,d)\in L_1$.
\begin{enumerate}
\item
$P(x)\equiv 1\mod 8$ if and only if one of the following holds;
\begin{itemize}
\item[(a)]$a,d\in\EE,\ b,c\in\OO$,
\item[(b)]$a,d\in\OO,\ b+c\in\OO$.
\end{itemize}
\item
$P(x)\equiv 5\mod 8$ if and only if one of the following holds;
\begin{itemize}
\item[(a)]$b,c\in\EE,\ a,d\in\OO$,
\item[(b)]$b,c\in\OO,\ a+d\in\OO$.
\end{itemize}
\end{enumerate}
\end{lem}
\begin{proof}
Let $P(x)\equiv 1\mod 4$.
Then $ad+bc\in\OO$ and $P(x)\equiv 1+4\Delta\mod 8$.
Hence to know $P(x)\mod 8$,
what we should see is $\Delta\mod 2$.
Now the lemma follows from the observations below.
In the following congruence expression means modulo $2$.
\begin{itemize}
\item[(I)]Assume $a+d\in\OO$.
Then $ad\in\EE,\, bc\in\OO,\, b,c\in\OO$.
Hence $\Delta\equiv ac^3+bd^3\equiv a+d\equiv 1$.
\item[(II)]Assume $a+d\in\EE$.
If $a,d\in\OO$, then $bc\in\EE$ and $\Delta\equiv b^3+c^3+1\equiv b+c+1$.
Hence either $(b,c\in\EE,\,\Delta\equiv1)$
or $(b+c\in\OO,\,\Delta\equiv0)$.
If $a,d\in\EE$, then $bc\in\OO$ and hence $\Delta\equiv 0$.
\end{itemize}
\vspace{-\baselineskip}
\end{proof}

\begin{proof}[Proof of Proposition \ref{pr:L7L9}]
We first show $L_7=2L_1\amalg L_{1,\equiv 1\,(8)}$.
Let $x=(a,b,c,d)\in L_{1,\equiv 1\,(8)}$.
Then by the lemma above we have $a+b+c,b+c+d\in\EE$
and so $x\in L_7$.
Hence $L_7\supset 2L_1\amalg L_{1,\equiv 1\,(8)}$.
We consider the reverse inclusion.
Let $x=(a,b,c,d)\in L_7$.
Then $a+b+c,b+c+d\in\EE$, and so $a+d\in\EE$.
First assume $a,d\in\OO$.
Then $b+c\in\OO$ and hence $x\in L_{1,\equiv 1\,(8)}$.
Next assume $a,d\in\EE$. Then $b+c\in\EE$ and hence
either $(a,b,c,d\in\EE)$ or $(a,d\in\EE,\,b,c\in\OO)$.
This shows $x\in 2L_1\amalg L_{1,\equiv 1\,(8)}$.
Hence $L_7\subset 2L_1\amalg L_{1,\equiv 1\,(8)}$.

The equation $L_9=2L_1\amalg L_{1,\equiv 5\,(8)}$
is proved similarly.
\end{proof}

We next consider the lattices in $L_2$.
Recall that for $x\in L_2$, $P(x)$ is a multiple of $27$.
We put $Q(x)=P(x)/27$.
Then
$Q(x)\equiv 3P(x)\mod 8$.
\begin{defn}
Let $L$ be a lattice in $L_2$.
For $l,N\in\Z$, $N\neq0$, we put
\[
L_{\equiv' l\,(N)}=\{x\in L\mid Q(x)\equiv l\mod N\}.
\]
\end{defn}
Since $Q(x)\equiv 3P(x)\mod 8$,
we have $L_{\equiv l\,(8)}=L_{\equiv' 3l\,(8)}$.
\begin{prop}\label{pr:L8L10}
We have
\begin{align*}
L_8&=2L_2\amalg L_{2,\equiv' 7\,(8)},\\
L_{10}&=2L_2\amalg L_{2,\equiv' 3\,(8)},
\end{align*}
\end{prop}
\begin{proof}
The first one follows from
$L_9=2L_1\amalg L_{1,\equiv 5\,(8)}$
we proved in Proposition \ref{pr:L7L9} and
\[
L_9\cap L_2=L_8,
\quad
2L_1\cap L_2=2L_2,
\quad
L_{1,\equiv 5\, (8)}\cap L_2=L_{2,\equiv5\,(8)}=L_{2,\equiv'7\,(8)}.
\]
The second one is proved similarly.
\end{proof}

We now give a proof of Theorem \ref{thm:extrafneqL7L10}.
\begin{proof}[Proof of Theorem \ref{thm:extrafneqL7L10}]
Let $\{a_n\}$ be the coefficients of
$\xi_-(L_1,s)$;
\[
\xi_-(L_1,s)=\sum_{n\geq1}\frac{a_n}{n^s}.
\]
Then by Proposition \ref{pr:L7L9},
\begin{align*}
\xi_-(L_7,s)
&=	\frac{1}{2^{4s}}\xi_-(L_1,s)
		+\sum_{n\geq1,n\equiv 7\,(8)}\frac{a_n}{n^s},\\
\xi_-(L_9,s)
&=	\frac{1}{2^{4s}}\xi_-(L_1,s)
		+\sum_{n\geq1,n\equiv 3\,(8)}\frac{a_n}{n^s}.
\end{align*}
If we put $\xi_+(L_2,s)=\sum_{n\geq1}b_n/n^s$
then similarly by Proposition \ref{pr:L8L10} we have
\begin{align*}
\xi_+(L_8,s)
&=	\frac{1}{2^{4s}}\xi_+(L_2,s)
		+\sum_{n\geq1,n\equiv 7\,(8)}\frac{b_n}{n^s},\\
\xi_+(L_{10},s)
&=	\frac{1}{2^{4s}}\xi_+(L_2,s)
		+\sum_{n\geq1,n\equiv 3\,(8)}\frac{b_n}{n^s}.
\end{align*}
Hence the first two formulas follows from 
$\xi_-(L_1,s)=\xi_+(L_2,s)$ and $a_n=b_n$.
The rests are proved similarly.
\end{proof}
We will give some properties on $\xi_{\pm}(L_i,s)$.
These can be checked using the table of the coefficients
of $\xi_{\pm}(L_i,s)$ given in Section \ref{sec:coefftable}.
\begin{prop}
\begin{enumerate}
\item
The Dirichlet series
$\xi_{\pm}(L_i,s)$ does not have an Euler product.
\item
The linear relations of the twenty Dirichlet series
$\{\xi_{\pm}(L_i,s)\}$ are exhausted by that given in
Theorems \ref{thm:extrafneqL1L2} and \ref{thm:extrafneqL7L10}.
Namely, the $\C$-vector space spanned by Dirichlet series
by $\{\xi_{\pm}(L_i,s)\}$ is of dimension $14$.
\end{enumerate}
\end{prop}

\section{Analytic properties of the Dirichlet series}
\label{sec:residue}
In this section, we study analytic properties of $\xi_\pm(L_i,s)$.
We also separate the contributions of irreducible binary cubic forms
and reducible binary cubic forms in the residue formulas.
Let $V_\Z^\ird=\{x(v)\in V_\Z\mid \text{$x(v)$ is irreducible over $\Q$}\}$
and $V_\Z^\rd=V_\Z\setminus V_\Z^\ird$. They are $G_\Z$-invariant subsets.

\begin{defn}\label{defn:Dirichletird}
\begin{enumerate}
\item
For $i=1,3,5,7,9$, we put
\[
\xi_\pm^\ird(L_i,s)=\sum_{x\in{G^1_\Z}\backslash (L_i^\pm\cap V_\Z^\ird)}
\frac{({}^{\#} G^1_{\Z,x})^{-1}}{|P(x)|^s},
\qquad
\xi_\pm^\rd(L_i,s)=\sum_{x\in{G^1_\Z}\backslash (L_i^\pm\cap V_\Z^\rd)}
\frac{({}^{\#} G^1_{\Z,x})^{-1}}{|P(x)|^s}.
\]
\item
For $i=2,4,6,8,10$, we put
\[
\xi_\pm^\ird(L_i,s)=3^{3s}\sum_{x\in{G^1_\Z}\backslash (L_i^\pm\cap V_\Z^\ird)}
\frac{({}^{\#} G^1_{\Z,x})^{-1}}{|P(x)|^s},
\qquad
\xi_\pm^\rd(L_i,s)=3^{3s}\sum_{x\in{G^1_\Z}\backslash (L_i^\pm\cap V_\Z^\rd)}
\frac{({}^{\#} G^1_{\Z,x})^{-1}}{|P(x)|^s}.
\]
\end{enumerate}
\end{defn}
By definition we have $\xi_\pm(L_i,s)=\xi_\pm^\ird(L_i,s)+\xi_\pm^\rd(L_i,s)$.
\begin{defn}
For $i=1,3,5,7,9$, we put
$a_i=[\widehat{L_i}:L_{i+1}]$ and $2^{b_i}=[V_\Z:L_i]$,
where $\widehat{L_i}$ is the dual lattice of $L_i$
with respect to the bilinear form $\langle x, y\rangle$.
\end{defn}
It is easy to see that
$(a_i,b_i)$ is
$(0,0)$, $(2,1)$, $(2,3)$, $(2,2)$, $(2,2)$
for $i=1,3,5,7,9$, respectively.
The analytic properties of these series are summarized as follows.
\begin{thm}\label{thm:analyticprop}
\begin{enumerate}
\item
The Dirichlet series $\xi_\pm(L_i,s)$
can be continued holomorphically to the whole complex plane
except for simple poles at $s=1$ and $5/6$.
Furthermore, they satisfy the following functional equations
\begin{equation*}
\begin{pmatrix}
\xi_+(L_i,1-s)\\
\xi_-(L_i,1-s)\\
\end{pmatrix}
=
\frac{2^{2a_is-b_i}3^{3s-2}}{2\pi^{4s}}
\Gamma(s)^2\Gamma(s-\frac16)\Gamma(s+\frac16)
\begin{pmatrix}
\sin 2\pi s&\sin \pi s\\
3\sin \pi s&\sin2\pi s\\
\end{pmatrix}
\begin{pmatrix}
\xi_+(L_{i+1},s)\\
\xi_-(L_{i+1},s)\\
\end{pmatrix}
\end{equation*}
where $i=1,3,5,7,9$.
\item
The Dirichlet series $\xi_\pm^\ird(L_i,s)$ and $\xi_\pm^\rd(L_i,s)$
have meromorphic continuations to the whole complex plane.
The first one is holomorphic for $\Re(s)>1/2$ except for
simple poles at $s=1$ and $s=5/6$.
The second one is holomorphic for $\Re(s)>1/2$ except for
a simple pole at $s=1$.
\item
Let
\begin{gather*}
\alpha_{i,\pm}=\res_{s=1}\xi_{\pm}(L_i,s),
\quad
\beta_{i,\pm}=\res_{s=5/6}\xi_{\pm}(L_i,s),\\
\alpha_{i,\pm}^\ird=\res_{s=1}\xi_{\pm}^\ird(L_i,s),
\quad
\alpha_{i,\pm}^\rd=\res_{s=1}\xi_{\pm}^\rd(L_i,s).
\end{gather*}
Then if we put
\[
\alpha		=\frac{\pi^2}{9},\qquad
\beta		=\frac{3^{1/2}(2\pi)^{1/3}}{18}
		\zeta\left(\frac23\right)
		\Gamma\left(\frac13\right)
		\Gamma\left(\frac23\right)^{-1},
\]
the values are given by Table \ref{table:residue}.
\end{enumerate}
\end{thm}
\begin{table}
\begin{align*}
\begin{array}{c||ccccc|ccccc}
i&1&3&5&7&9&2&4&6&8&10\\
\hline\hline
&&&&&&&&&\\
\alpha_{i,+}
&	\alpha
&	\dfrac{\alpha}{2}
&	\dfrac{7}{32}\alpha
&	\dfrac{\alpha}{4}
&	\dfrac{\alpha}{4}
&	\dfrac32\alpha
&	\dfrac{9}{32}\alpha
&	\dfrac34{\alpha}
&	\dfrac38{\alpha}
&	\dfrac38{\alpha}
\\&&&&&&&&&\\
\beta_{i,+}
&	\beta
&	\dfrac{\beta}{2}
&	\dfrac{\beta}{4\sqrt[3]2}
&	\dfrac{\beta}{4}
&	\dfrac{\beta}{4}
&	\sqrt3\beta
&	\dfrac{\sqrt3}{4\sqrt[3]2}\beta
&	\dfrac{\sqrt3}2\beta
&	\dfrac{\sqrt3}{4}\beta
&	\dfrac{\sqrt3}{4}\beta
\\&&&&&&&&&\\
\alpha_{i,+}^\ird
&	\dfrac14\alpha
&	\dfrac18\alpha
&	\dfrac1{32}\alpha
&	\dfrac1{16}\alpha
&	\dfrac1{16}\alpha
&	\dfrac34\alpha
&	\dfrac3{32}\alpha
&	\dfrac38\alpha
&	\dfrac3{16}\alpha
&	\dfrac3{16}\alpha
\\&&&&&&&&&\\
\alpha_{i,+}^\rd
&	\dfrac34\alpha
&	\dfrac38\alpha
&	\dfrac3{16}\alpha
&	\dfrac3{16}\alpha
&	\dfrac3{16}\alpha
&	\dfrac34\alpha
&	\dfrac3{16}\alpha
&	\dfrac38\alpha
&	\dfrac3{16}\alpha
&	\dfrac3{16}\alpha
\\&&&&&&&&&\\\hline&&&&&&&&&\\
\alpha_{i,-}
&	\dfrac32\alpha
&	\dfrac34\alpha
&	\dfrac9{32}\alpha
&	\dfrac38\alpha
&	\dfrac38\alpha
&	3	\alpha
&	\dfrac{15}{32}\alpha
&	\dfrac32\alpha
&	\dfrac34\alpha
&	\dfrac34\alpha
\\&&&&&&&&&\\
\beta_{i,-}
&	\sqrt3\beta
&	\dfrac{\sqrt3}{2}\beta
&	\dfrac{\sqrt3}{4\sqrt[3]2}\beta
&	\dfrac{\sqrt3}{4}\beta
&	\dfrac{\sqrt3}{4}\beta
&	3\beta
&	\dfrac{3}{4\sqrt[3]2}\beta
&	\dfrac32\beta
&	\dfrac{3}{4}\beta
&	\dfrac{3}{4}\beta
\\&&&&&&&&&\\
\alpha_{i,-}^\ird
&	\dfrac34\alpha
&	\dfrac38\alpha
&	\dfrac3{32}\alpha
&	\dfrac3{16}\alpha
&	\dfrac3{16}\alpha
&	\dfrac94\alpha
&	\dfrac9{32}\alpha
&	\dfrac98\alpha
&	\dfrac9{16}\alpha
&	\dfrac9{16}\alpha
\\&&&&&&&&&\\
\alpha_{i,-}^\rd
&	\dfrac34\alpha
&	\dfrac38\alpha
&	\dfrac3{16}\alpha
&	\dfrac3{16}\alpha
&	\dfrac3{16}\alpha
&	\dfrac34\alpha
&	\dfrac3{16}\alpha
&	\dfrac38\alpha
&	\dfrac3{16}\alpha
&	\dfrac3{16}\alpha
\\
\end{array}
\end{align*}
\caption{}%
\label{table:residue}
\end{table}

\begin{proof}
For $L_1$ and $L_2$, Shintani \cite{shintania, shintanib}
proved this theorem by establishing
the theory of zeta functions associated with
the space of binary cubic forms and the space of binary quadratic forms.
His global theory was rewritten in the adelic language
by Wright \cite{wright} and the second author \cite{taniguchi}.
We would like to mention that a quite simpler version of the global theory
for the space of binary cubic forms \cite{wright}
were given by Kogiso \cite{kogiso}.
Let $\A$ and $\A_\fin$ be the rings of adeles and finite adeles
of $\Q$, respectively. Note that $\A_\fin=\widehat\Z\otimes_\Z\Q$
and $\A=\A_\fin\times\R$,
where $\widehat\Z$ is the profinite completion of $\Z$.
Let $\cs(V_{\A})$, $\cs(V_{\A_\fin})$ and $\cs(V_\R)$
be the spaces of Schwartz-Bruhat functions on each of the indicated domains.
Let $\Phi_\fin\in\cs(V_{\A_\fin})$ be the characteristic
function of $L_i\otimes_\Z\widehat\Z\subset V_{\A_\fin}$
and $\Phi_\infty\in\cs(V_\R)$ arbitrary.
Then by considering the global zeta functions in \cite{taniguchi, wright}
with the test function $\Phi_\fin\otimes\Phi_\infty\in\cs(V_\A)$,
we can prove the theorem the same way as \cite{shintania, shintanib}.
Here we illustrate the proof of (3) with $i=3,5,7,9$.
We fix a prime $p$. We fix any Haar measures $du$ on $\Q_p$
and $d^\times t$ on $\Q_p^\times$.
For $t\in \Q_p^\times$, we put $|t|_p=d(tu)/du$.
For $\Phi\in\cs(V_{\Q_p})$, we define
\begin{align*}
\esA_p^\ird(\Phi)
	&=\int_{\Q_p^4}\Phi(u_1,u_2,u_3,u_4)du_1du_2du_3du_4,\\
\esA_p^\rd(\Phi)
	&=\int_{\Q_p^\times\times \Q_p^2}
		|t|_p^2\Phi(0,t,u_1,u_2)d^\times tdu_1du_2,\\
\esB_p(\Phi)
	&=\int_{\Q_p^\times\times \Q_p^3}
		|t|_p^{1/3}\Phi(t,u_1,u_2,u_3)d^\times tdu_1du_2du_3.
\end{align*}
Let $\Phi_i$ be the characteristic function of $L_i\otimes\Z_p$.
Since $i=3,5,7,9$ we have $\Phi_i=\Phi_1$ unless $p=2$.
Hence by \cite[Proposition 8.6]{taniguchi}, we have
\[
\frac{\alpha_{i,\pm}^\ird}{\alpha_{1,\pm}^\ird}
	=\frac{\esA_2^\ird(\Phi_i)}{\esA_2^\ird(\Phi_1)},
\qquad
\frac{\alpha_{i,\pm}^\rd}{\alpha_{1,\pm}^\rd}
	=\frac{\esA_2^\rd(\Phi_i)}{\esA_2^\rd(\Phi_1)},
\qquad
\frac{\beta_{i,\pm}}{\beta_{1,\pm}}
	=\frac{\esB_2(\Phi_i)}{\esB_2(\Phi_1)}.
\]
The computations of the right hand sides in the equations
are easily carried out. For example,
\[
\frac{\esA_2^\ird(\Phi_3)}{\esA_2^\ird(\Phi_1)}=\frac12,
\qquad
\frac{\esA_2^\rd(\Phi_5)}{\esA_2^\rd(\Phi_1)}=\frac14,
\qquad
\frac{\esB_2(\Phi_7)}{\esB_2(\Phi_1)}=\frac14.
\]
Since $\alpha_{1,\pm}^\ird$, $\alpha_{1,\pm}^\rd$
and $\beta_{1,\pm}$ are known, we obtain the value.
Note that $\alpha_{i,\pm}=\alpha_{i,\pm}^\ird+\alpha_{i,\pm}^\rd$.
The rest are proved similarly and we omit the detail.
Note that $a_3=a_5=a_7=a_9=2$ in (1) comes from the fact that
for $i=3,5,7,9$ the dual lattice of $L_i$
with respect to the alternating form on $V$
is $2^{-1}L_{i+1}$.
Also $b_3=1$, $b_5=3$ and $b_7=b_9=2$ are because
$[L_1:L_3]=2$, $[L_1:L_5]=8$ and $[L_1:L_7]=[L_1:L_9]=4$,
respectively.
\end{proof}

\begin{rem}
As in \cite[Proposition 2.1]{ohno},
the functional equation in the theorem is compatible with
Theorem \ref{thm:extrafneqL7L10}. For example,
from $\xi_-(L_7,s)=\xi_+(L_8,s)$ and Theorem \ref{thm:analyticprop}
(1) for $i=7$, we can deduce $\xi_-(L_8,s)=3\xi_+(L_7,s)$.
\end{rem}

We discuss on the diagonalization
of the functional equation in Theorem \ref{thm:analyticprop} (1)
following \cite[Proposition 4.1]{dawra} and a related important
observation given in \cite[p.1088]{ohno}.
Let $a_{i+1}=a_i$ for $i=1,3,5,7,9$.
\begin{defn}\label{defn:lambda}
For $1\leq i\leq 10$ and each sign $\pm$, we put
\[
\Lambda_\pm(L_i,s)
	=\frac{2^{(a_i+1)s}3^{3s/2}}{\pi^{2s}}
		\Gamma(s)\Gamma(\frac{s}{2}+\frac14\mp\frac16)
		\Gamma(\frac{s}{2}+\frac14\mp\frac13)
	\left(\sqrt3\xi_+(L_i,s)\pm\xi_-(L_i,s)\right).
\]
\end{defn}
As a corollary to Theorem \ref{thm:analyticprop}
we have the following.
\begin{cor}\label{cor:DWfneq}
\begin{enumerate}
\item
For $i=1,3,5,7,9$,
\[
\Lambda_\pm(L_{i},1-s)=\pm3^{-1/2}2^{a_i-b_i}\Lambda_\pm(L_{i+1},s).
\]
\item Let $1\leq i\leq 10$. The function
$\Lambda_+(L_i,s)$ is holomorphic
except for simple poles at $s=0,1/6,5/6,1$,
while $\Lambda_-(L_i,s)$ is holomorphic
except for simple poles at $s=0,1$.
\item
Let $1\leq i\leq 10$.
The set of zeros of the
Dirichlet series $\sqrt3\xi_+(L_i,s)+\xi_-(L_i,s)$
and $\sqrt3\xi_+(L_i,s)-\xi_-(L_i,s)$
in the negative real axis are respectively given by
\begin{gather*}
\{-n\mid n\in\Z_{\geq1}\}\cup\{-2n+1/6\mid n\in\Z_{\geq1}\}\cup\{-2n+11/6\mid n\in\Z_{\geq1}\},\\
\{-n\mid n\in\Z_{\geq1}\}\cup\{-2n+5/6\mid n\in\Z_{\geq1}\}\cup\{-2n+7/6\mid n\in\Z_{\geq1}\},
\end{gather*}
where we put $\Z_{\geq1}=\{n\in\Z\mid n\geq1\}$.
\end{enumerate}
\end{cor}
\begin{proof}
By a simple computation we can prove that
the equalities in (1) are equivalent to
the functional equation given in
Theorem \ref{thm:analyticprop} (1).
(2) follows from
the values of residues given in
Theorem \ref{thm:analyticprop} (3)
and equalities (1) of this corollary.
(3) follows from (2) and Definition \ref{defn:lambda}.
\end{proof}

It is interesting that the poles of $\Lambda_-(L_i,s)$
at $s=5/6$ vanishes.
Taking the properties in Corollary \ref{cor:DWfneq}
into account,
it may be natural to ask that whether the Dirichlet series
$\sqrt3\xi_+(L_i,s)\pm\xi_-(L_i,s)$ has an Euler product.
The answer is negative.

\begin{prop}\label{prop:EulerProduct}
None of the Dirichlet series $\sqrt3\xi_+(L_i,s)+\xi_-(L_i,s)$,
$\sqrt3\xi_+(L_i,s)-\xi_-(L_i,s)$ $(1\leq i\leq 10)$ has an Euler product.
\end{prop}
\begin{proof}
If a Dirichlet series $\sum_{n\geq1}c_n/n^s$ has an Euler product,
then $c_1c_{pq}=c_pc_q$ for any distinct primes $p$ and $q$.
We can immediately confirm 
that any of our Dirichlet series does not satisfy this relation
for $p=3$ and $q=5$
using the table given in Section \ref{sec:coefftable}.
\end{proof}

Now we assume $i=1,7,9$.
Then Theorems \ref{thm:extrafneqL1L2}, \ref{thm:extrafneqL7L10}
assert
$\Lambda_\pm(L_{i+1},s)=\pm\sqrt{3}\Lambda_\pm(L_{i},s)$.
Since $a_i=b_i$ also, the functional equation
in Corollary \ref{cor:DWfneq} (1)
turns out to be of a single function
$\Lambda_\pm(L_{i},s)$.

\begin{thm}\label{cor:singlefneq}
For $i=1,2,7,8,9,10$,
\[
\Lambda_\pm(L_{i},1-s)=\Lambda_\pm(L_{i},s).
\]
\end{thm}
Namely, for $i=1,2,7,8,9,10$,
the function
$\Lambda_\pm(L_{i},s)$
is invariant if we replace $s$ by $1-s$.

We conclude this section
with deriving asymptotic behavior of some arithmetic functions.
For $n\in\Z, n\neq0$, let
$h(L_i,n)$ be the number of
$G^1_\Z$-orbit in $L_i\cap V_\Z^\ird$
with discriminant $n$.
Applying Tauberian theorem, Shintani \cite[Theorem 4]{shintanib}
obtained an asymptotic formula of the function $\sum_{0<\pm n<X}h(L_1,n)$.
By the same argument, we have the following.
Note that the functional equations of
$\xi_{\pm}(L_i,s)$ and $\xi_{\pm}^\rd(L_i,s)$ are used in the proof.
\begin{thm}\label{thm:density}
\begin{enumerate}
\item
Let $i$ be either $1,3,5,7$ or $9$. For any $\varepsilon>0$, 
\[
\sum_{0<\pm n<X}h(L_i,n)
	=\alpha_{i,\pm}^\ird X
	+\frac{\beta_{i,\pm}}{5/6}X^{5/6}
	+O(X^{2/3+\varepsilon})
	\qquad (X\to\infty).
\]
\item
Let $i$ be either $2,4,6,8$ or $10$. For any $\varepsilon>0$, 
\[
\sum_{0<\pm n<X}h(L_i,27n)
	=\alpha_{i,\pm}^\ird X
	+\frac{\beta_{i,\pm}}{5/6}X^{5/6}
	+O(X^{2/3+\varepsilon})
	\qquad (X\to\infty).
\]
\end{enumerate}
\end{thm}

\section{Proof of Theorem \ref{thm:classifyinvlat}}
\label{sec:prooflattice}

In this section, we prove Theorem \ref{thm:classifyinvlat}. 
We use an argument similar to \cite[Section 3]{is}. 
Let $L$ be a ${\rm SL}_2(\Z)$-invariant lattice. 
By taking some constant multiple if necessary, we can assume that $L$ is contained in $L_1$ and that there exists an element $x\in L$ such that $p^{-1}x\not\in L_1$ for each prime $p$. 
Such an element $x$ is called primitive for $p$. 
We put $(L)_p=L\otimes_{\Z} \Z_p$ for a prime $p$. 
In the following, we prove that
$(L)_p=(L_1)_p$ $(p\neq 2,3)$ in Lemma \ref{lem:l1},
$(L)_3=(L_1)_3$ or $(L_2)_3$ in Lemma \ref{lem:l2}, and
$(L)_2=(L_1)_2$, $(L_3)_2$, $(L_5)_2$, $(L_7)_2$ or $(L_9)_2$
in Lemma \ref{lem:l4}. 
It is easy to see that the lattices $L_1,L_2,\ldots,L_{10}$ are $G^1_{\Z}$-invariant.
Therefore we get Theorem \ref{thm:classifyinvlat} by these facts, because $L=\cap_{p:\mathrm{prime}} ( V_{\Q} \cap (L)_p ) $. 

From now, we shall prove Lemmas \ref{lem:l1}, \ref{lem:l2} and \ref{lem:l4}. 
Since ${\rm SL}_2(\Z_p)$ contains ${\rm SL}_2(\Z)$ as a dense subgroup,
$(L)_p$ is ${\rm SL}_2(\Z_p)$-invariant.
We put
\[
u(\alpha) = \begin{pmatrix} 1 & \alpha \\ 0 & 1 \end{pmatrix},\quad
w=  \begin{pmatrix} 0 & 1 \\ -1 & 0 \end{pmatrix}
\]
and $E_1=( 1, 0, 0 , 0)$, $E_2= ( 0, 1, 0, 0 )$, $ E_3= ( 0, 0, 1, 0)$, $E_4= ( 0, 0, 0, 1)$.
The action of $u(\alpha)$ on $L$ is given by
\[
u(\alpha)\cdot x
=(
	x_1+\alpha x_2+\alpha^2 x_3+\alpha^3 x_4,
	x_2+2\alpha x_3+ 3\alpha^2 x_4,
	x_3+3\alpha x_4,
	x_4
).
\]
For $x\in L$, we put
\[
\psi(x)=u(1)\cdot x-x=(x_2+x_3+x_4,2x_3+3x_4,3x_4,0)\in L.
\]

\begin{lem}\label{lem:l1} 
If $p\not =2,3$, then $(L)_p=(L_1)_p$.
\end{lem}
\begin{proof}
Let $x=(x_1,x_2,x_3,x_4)\in L$ be primitive for $p$. 

First we assume that $x_1\in \Z_p^{\times}$ or $x_4\in \Z_p^{\times}$. 
By considering the action of $w$, we may assume $x_4\in \Z_p^{\times}$.
Let $X_1=x_4^{-1}u(-3^{-1}x_4^{-1}x_3 ) \cdot x$.
Then since $X_1$ is of the form  $(*,*,0,1)$,
we have $6^{-1}\psi(\psi(X_1))=(1,1,0,0)$.
Since $E_2=u(-1)\cdot (1,1,0,0)$ and $E_1=\psi(E_2)$,
we have $E_1$, $E_2$, $E_3$, $E_4 \in (L)_p$. 
Hence $(L)_p=(L_1)_p$. 

Second we assume $x_1, x_4\not\in \Z_p^{\times}$. 
Then we have $x_2\in \Z_p^{\times}$ or $x_3\in \Z_p^{\times}$. 
We may assume $x_3\in \Z_p^{\times}$.
Since the first component of $u(1)\cdot x+u(-1)\cdot x-2x$
is $2x_3\in\Z_p^{\times}$,
by the argument above we have $(L)_p=(L_1)_p$. 
\end{proof}

\begin{lem}\label{lem:l2} 
$(L)_3=(L_1)_3$ or $(L_2)_3$.
\end{lem}
\begin{proof}
Let $x=(x_1,x_2,x_3,x_4)\in L$ be primitive for $3$. 

First we assume $x_2\in \Z_3^{\times}$ or $x_3\in \Z_3^{\times}$. 
Taking the action of $w$ into account, we may assume $x_3\in \Z_3^{\times}$.
Let
$X_1=(2x_3+3x_4)^{-1}\psi(x)= ( x_1' , 1 , x_3' , 0 )$
and
$X_2 = (2x_3+6x_4)^{-1}\psi(\psi(x))= ( 1 , x_2' , 0 , 0 )$.
Then $x_2', x_3'\in 3\Z_3$.
Further we put
$X_3=X_1 - x_1'X_2 =( 0 , 1-x_1'x_2' , x_3' , 0 )$, $1-x_1'x_2'\in \Z_3^{\times}$, 
$X_4=(1-x_1'x_2')^{-1}(w \cdot X_3) = ( 0 , x_2'' , 1 , 0 )$, $x_2''\not\in \Z_3^{\times}$, 
$X_5 = u(-2^{-1}x''_2) \cdot X_4 = ( x_1'' , 0 , 1 , 0 )$, 
$X_6=\psi(X_5) = ( 1 , 2 , 0 , 0) $.
Then since $E_1=2^{-1}\psi(X_6)$ and 
$E_2=2^{-1}(X_6-E_1)$,
we have $(L)_3=(L_1)_3$. 

Second we assume $x_2,x_3\not\in \Z_3^{\times}$. 
Then we have $x_1\in \Z_3^{\times}$ or $x_4\in \Z_3^{\times}$. 
We may assume $x_4\in \Z_3^{\times}$. 
We have 
$X_7=\psi(x)= ( x_2+x_3+x_4 , 2x_3+3x_4 , 3x_4 , 0 )$, $x_2+x_3+x_4\in \Z_3^{\times}$, $2x_3+3x_4\in 3\Z_3$, $3x_4\in 3\Z_3^{\times}$, 
$X_8 = u(-2^{-1}x_4^{-1}\cdot 3^{-1}(2x_3+3x_4) )\cdot X_7=( x_1' , 0 , 3x_4 , 0 )$, $x_1' \in \Z_3^{\times}$, $3x_4\in 3\Z_3^{\times}$.
Then since
$x_4^{-1}\psi(X_8)= 3 E_1  + 6E_2$, 
$2^{-1}\psi( 3 E_1  + 6E_2  )= 3E_1$, 
$3E_2= 2^{-1}( (3 E_1  + 6E_2)-3E_1  )$ and
$E_1=x_1^{\prime -1}\cdot (X_8-3x_4E_3)$,
we get $(L_2)_3 \subset (L)_3$. 

We see $(L_2)_3 \subset (L)_3 \subset (L_1)_3$ from the above results. 
Suppose $(L_2)_3 \neq (L)_3$. 
Since $ (L_1)_3 / (L_2)_3$ is represented by the set $\left\{ aE_2+bE_3 \,\, ; \,\, 0\leq a,b \leq 2 \right\}$, 
$(L)_3$ has an element of the form $aE_2+bE_3$ for some $(a,b)\neq (0,0)$.
Hence we have $(L)_3=(L_1)_3$. 
So we get this lemma. 
\end{proof}

\begin{lem}\label{lem:l3}
$(L)_2$ contains $(L_5)_2$ or $(L_9)_2$.
\end{lem}
\begin{proof}
Let $x=(x_1,x_2,x_3,x_4)\in L$ be primitive for $2$.

(i) We assume $x_1\in \Z_2^{\times}$ or $x_4\in \Z_2^{\times}$. 
We may assume $x_4\in \Z_2^{\times}$. 
Let
$X_1=u(-3^{-1}x_4^{-1}x_3) \cdot x= ( * , * , 0 , x_4 )$, 
$X_2=(3x_4)^{-1}\psi(X_1)= ( x_1' , 1  , 1 , 0 )$.
Then since $2E_1+2E_2 =\psi(X_2)$ and
$2E_1=\psi(\psi(X_2))$,
we have $2E_1$, $2E_2$, $2E_3$, $2E_4 \in (L)_2$. \\
(i-a) We assume $x_1' \not\in \Z_2^{\times} $. 
We have $E_2+E_3=X_2- (2^{-1} x_1') \cdot (2E_1) \in (L)_2$. 
Since $L_5=\Z (2E_1) + \Z (2E_4) + \Z (E_2+E_3) + \Z (2E_2) $,
we get $(L_5)_2 \subset (L)_2$. \\
(i-b) We assume $x_1' \in \Z_2^{\times} $. 
From $x_1'=1+x_1''$, $(x_1''\in 2 \Z_2 )$, we have $X_2- (2^{-1}x_1'')\cdot (2E_1)=E_1+E_2+E_3 $. 
Since $L_9= \Z (E_1+E_2+E_3) +  \Z (E_2+E_3+E_4) + \Z (2E_1) + \Z (2E_2)$.
we get $(L_9)_2 \subset (L)_2  $.

(ii) We assume $x_1,x_4 \not\in \Z_2^{\times}$. \\
(ii-a) We assume $x_2+x_3\in \Z_2^{\times}$. 
Since the first component of $\psi(x)$
is $x_2+x_3+x_4\in \Z_2^{\times}$,
we can reduce the case (ii-a) to the case (i). \\
(ii-b) We assume $x_2+x_3 \not\in \Z_2^{\times}$. 
Since $x$ is primitive, we have $x_2,\, x_3\in \Z_2^{\times}$. 
We have 
$X_3=(x_3+3x_4)^{-1}\psi(\psi(x))=( 2 , c , 0 , 0 )$, $c\in 4\Z_2$, 
$X_4=w^{-1}\cdot X_3=-cE_3+2E_4$. 
Furthermore we put
$X_5= x - (2^{-1}x_1)\cdot X_3 - (2^{-1}x_4) \cdot X_4=(0,\alpha,\beta,0)$.
Then $\alpha=x_2-2^{-1}x_1c\in\Z_2^\times$,
$\beta=x_3+2^{-1}x_4c\in \Z_2^{\times}$.
Let
$X_6=\psi(X_5) - 2^{-1}(\alpha+\beta)X_3= ( 0 , 2\beta - 2^{-1}(\alpha+\beta)c  , 0 , 0 ) $. Then $ 2\beta - 2^{-1}(\alpha+\beta ) c \in 2\Z_2^{\times}$. 
Hence we have 
$ 2 E_2 =  (\beta - 2^{-2}(\alpha+\beta)c )^{-1}X_6$, 
$ 2E_1 = X_3 - (2^{-1}c)\cdot (2E_2)$, $2E_3$, $2E_4\in (L)_2$, 
$E_2+E_3 = X_5 - 2^{-1}(\alpha-1)\cdot (2E_2) - 2^{-1}(\beta - 1)\cdot (2E_3) \in (L)_2 $. 
Therefore we get $(L_5)_2 \subset (L)_2$.
\end{proof}

\begin{lem}\label{lem:l4}
$(L)_2=(L_1)_2$, $(L_3)_2$, $(L_5)_2$, $(L_7)_2$ or $(L_9)_2$.
\end{lem}

\begin{proof}
Form Lemma \ref{lem:l3}, we know $(L_5)_2 \subset (L)_2\subset (L_1)_2 $ or $(L_9)_2 \subset (L)_2 \subset (L_1)_2 $. 
Hence we have only to take all representation elements of $(L_1)_2 / (L_5)_3 $, $(L_1)_2 / (L_9)_3 $ and compute all cases for subspaces containing representation elements.

(I) We treat the case $(L_5)_2 \subset (L)_2\subset (L_1)_2 $. 
Let $(L)_2\neq (L_5)_2$. 
Since $L_1=\Z E_1 + \Z E_4 + \Z (E_2+E_3) + \Z E_2$
and $L_5 = \Z (2E_1) + \Z (2E_4) + \Z (E_2+E_3) + \Z (2E_2) $,
$(L_1)_2 / (L_5)_2$ is represented by the set
$\{ aE_1+bE_4+cE_2 \,  ; \, 0 \leq a,b,c\leq 1  \}$. \\
(I-1) $(L)_2$ contains one of $E_1$, $E_4$, $E_1+E_4$. 
We easily see that $(L)_2$ contains $(L_3)_2$. 
Since $(L_1)_2 / (L_3)_2 \cong\Z/2\Z$,
$(L)_2$ is either$(L_1)_2$ or $(L_3)_2$. \\
(I-2) $(L)_2$ contains either $E_2$, $E_1+E_2$ or $E_2+E_4$.
Since $E_1=\psi(E_2)=\psi(E_1+E_2)=\psi(w\cdot (E_2+E_4))-2E_2$,
we have $(L)_2=(L_1)_2$. \\
(I-3) $(L)_2$ contains $E_1+E_2+E_4$. 
Since $L_7 = \Z (E_1+E_2+E_4) + \Z (E_1+E_3+E_4) + \Z (2E_1) + \Z (2E_4)$, we see $ (L_7)_2 \subset (L)_2 $. 
Furthermore $(L_1)_2 / (L_7)_2$ is represented by
$\{ 0, E_1, E_4, E_1+E_4 \}$. 
If $(L)_2$ contains one of this representation element,
then we have $(L)_2 = (L_1)_2$. 
Therefore $(L)_2=(L_1)_2$ or $(L_7)_2$.

(II) We treat the case $(L_9)_2 \subset (L)_2\subset (L_1)_2 $. 
Suppose $(L)_2\neq (L_9)_2$. 
Since
$L_1=\Z E_1 + \Z E_2 + \Z (E_1+E_2+E_3) + \Z (E_2+E_3+E_4)$
and $L_9 = \Z (E_1+E_2+E_3) + \Z (E_2+E_3+E_4) + \Z (2E_1) + \Z (2E_2) $,
$(L_1)_2 / (L_9)_2$
is represented by $\{ aE_1+bE_2 \, ; \,  0\leq a,b\leq 1 \}$. \\
(II-1) $(L)_2$ contains $E_1$. 
We have $ (L_3)_2 \subset (L)_2 $. 
Hence we have $(L)_2=(L_1)_2$ or $(L_3)_2$. \\
(II-2) $(L)_2$ contains $E_2$ or $E_1+E_2$. 
Since $\psi(E_2)=\psi(E_1+E_2)=E_1$, we have $ (L_1)_2 = (L)_2 $. 

Form (I) and (II), we get this lemma.
\end{proof}

\section{Table of the coefficients}\label{sec:coefftable}
We give the table of about first fifty coefficients
of the Dirichlet series $\xi_{\pm}(L_i,s)$.
In the table, we give the value multiplied by $3$
for the each coefficient except for
$\xi_+(L_{i},s)$, $i=2,4,6,8,10$
where in which cases
we give the exact value of the coefficients.
Hence the table means, for example,
\begin{align*}
\xi_+(L_4,s)
&=\frac{1/3}{3^s}+\frac1{11^s}+\frac1{19^s}
	+\frac{4/3}{27^s}+\frac1{35^s}
	+\frac1{43^s}+\frac1{48^s}+\frac1{51^s}+\dots,\\
\xi_-(L_7,s)
&=\frac{1/3}{1^s}+\frac1{9^s}+\frac{1/3}{16^s}
	+\frac1{17^s}+\frac1{25^s}
	+\frac1{33^s}+\frac1{41^s}+\frac{5/3}{49^s}+\dots,\\
\xi_+(L_8,s)
&=\frac{1}{1^s}+\frac3{9^s}+\frac1{16^s}
	+\frac{3}{17^s}+\frac3{25^s}
	+\frac3{33^s}+\frac3{41^s}+\frac5{49^s}+\dots.
\end{align*}
\begin{center}
\begin{tabular}{|r|cc|cccc|cccc|}
\hline
&\!\!\! $L_1^{-}$\!\!\!
&\!\!\! $L_2^{+}$\!\!
&\!\!   $L_3^{-}$\!\!\!
&\!\!\! $L_4^{+}$\!\!\!
&\!\!\! $L_5^{-}$\!\!
&\!\!\! $L_6^{+}$\!\!\!
&\!\!   $L_7^{-}$\!\!\!
&\!\!\! $L_8^{+}$\!\!\!
&\!\!\! $L_9^{-}$\!\!\!
&\!\!\! $L_{10}^{+}$\!\!
\\
\hline
3 & 3 & 3 & 3 & 1 & 0 & 1 & 0 & 0 & 3 & 3 \\
4 & 3 & 3 & 0 & 0 & 0 & 0 & 0 & 0 & 0 & 0 \\
7 & 3 & 3 & 3 & 0 & 3 & 3 & 3 & 3 & 0 & 0 \\
8 & 3 & 3 & 0 & 0 & 0 & 0 & 0 & 0 & 0 & 0 \\
11 & 3 & 3 & 3 & 3 & 0 & 3 & 0 & 0 & 3 & 3 \\
12 & 3 & 3 & 0 & 0 & 0 & 0 & 0 & 0 & 0 & 0 \\
15 & 3 & 3 & 3 & 0 & 3 & 3 & 3 & 3 & 0 & 0 \\
16 & 6 & 6 & 3 & 0 & 0 & 3 & 0 & 0 & 0 & 0 \\
19 & 3 & 3 & 3 & 3 & 0 & 3 & 0 & 0 & 3 & 3 \\
20 & 3 & 3 & 0 & 0 & 0 & 0 & 0 & 0 & 0 & 0 \\
23 & 9 & 9 & 3 & 0 & 3 & 9 & 9 & 9 & 0 & 0 \\
24 & 3 & 3 & 0 & 0 & 0 & 0 & 0 & 0 & 0 & 0 \\
27 & 6 & 6 & 6 & 4 & 0 & 4 & 0 & 0 & 6 & 6 \\
28 & 9 & 9 & 0 & 0 & 0 & 6 & 0 & 0 & 0 & 0 \\
31 & 9 & 9 & 3 & 0 & 3 & 9 & 9 & 9 & 0 & 0 \\
32 & 6 & 6 & 3 & 0 & 0 & 3 & 0 & 0 & 0 & 0 \\
35 & 3 & 3 & 3 & 3 & 0 & 3 & 0 & 0 & 3 & 3 \\
36 & 3 & 3 & 0 & 0 & 0 & 0 & 0 & 0 & 0 & 0 \\
39 & 3 & 3 & 3 & 0 & 3 & 3 & 3 & 3 & 0 & 0 \\
40 & 3 & 3 & 0 & 0 & 0 & 0 & 0 & 0 & 0 & 0 \\
43 & 3 & 3 & 3 & 3 & 0 & 3 & 0 & 0 & 3 & 3 \\
44 & 9 & 9 & 6 & 0 & 0 & 0 & 0 & 0 & 0 & 0 \\
47 & 3 & 3 & 3 & 0 & 3 & 3 & 3 & 3 & 0 & 0 \\
48 & 6 & 6 & 3 & 3 & 3 & 3 & 3 & 3 & 3 & 3 \\
51 & 3 & 3 & 3 & 3 & 0 & 3 & 0 & 0 & 3 & 3 \\
\hline
\end{tabular}\qquad
\begin{tabular}{|r|cc|cccc|cccc|}
\hline
&\!\!\! $L_1^{+}$\!\!\!
&\!\!\! $L_2^{-}$\!\!
&\!\!   $L_3^{+}$\!\!\!
&\!\!\! $L_4^{-}$\!\!\!
&\!\!\! $L_5^{+}$\!\!
&\!\!\! $L_6^{-}$\!\!\!
&\!\!   $L_7^{+}$\!\!\!
&\!\!\! $L_8^{-}$\!\!\!
&\!\!\! $L_9^{+}$\!\!\!
&\!\!\! $L_{10}^{-}$\!\!
\\
\hline
1 & 1 & 1 & 1 & 0 & 1 & 1 & 1 & 1 & 0 & 0 \\
4 & 3 & 3 & 0 & 0 & 0 & 2 & 0 & 0 & 0 & 0 \\
5 & 3 & 3 & 3 & 1 & 0 & 1 & 0 & 0 & 3 & 3 \\
8 & 3 & 3 & 0 & 0 & 0 & 0 & 0 & 0 & 0 & 0 \\
9 & 3 & 3 & 3 & 0 & 3 & 3 & 3 & 3 & 0 & 0 \\
12 & 3 & 3 & 0 & 0 & 0 & 0 & 0 & 0 & 0 & 0 \\
13 & 3 & 3 & 3 & 1 & 0 & 1 & 0 & 0 & 3 & 3 \\
16 & 4 & 4 & 1 & 1 & 1 & 3 & 1 & 1 & 1 & 1 \\
17 & 3 & 3 & 3 & 0 & 3 & 3 & 3 & 3 & 0 & 0 \\
20 & 3 & 3 & 0 & 0 & 0 & 0 & 0 & 0 & 0 & 0 \\
21 & 3 & 3 & 3 & 1 & 0 & 1 & 0 & 0 & 3 & 3 \\
24 & 3 & 3 & 0 & 0 & 0 & 0 & 0 & 0 & 0 & 0 \\
25 & 3 & 3 & 3 & 0 & 3 & 3 & 3 & 3 & 0 & 0 \\
28 & 3 & 3 & 0 & 0 & 0 & 0 & 0 & 0 & 0 & 0 \\
29 & 3 & 3 & 3 & 1 & 0 & 1 & 0 & 0 & 3 & 3 \\
32 & 6 & 6 & 3 & 0 & 0 & 3 & 0 & 0 & 0 & 0 \\
33 & 3 & 3 & 3 & 0 & 3 & 3 & 3 & 3 & 0 & 0 \\
36 & 9 & 9 & 0 & 0 & 0 & 6 & 0 & 0 & 0 & 0 \\
37 & 3 & 3 & 3 & 3 & 0 & 3 & 0 & 0 & 3 & 3 \\
40 & 3 & 3 & 0 & 0 & 0 & 0 & 0 & 0 & 0 & 0 \\
41 & 3 & 3 & 3 & 0 & 3 & 3 & 3 & 3 & 0 & 0 \\
44 & 3 & 3 & 0 & 0 & 0 & 0 & 0 & 0 & 0 & 0 \\
45 & 3 & 3 & 3 & 1 & 0 & 1 & 0 & 0 & 3 & 3 \\
48 & 6 & 6 & 3 & 0 & 0 & 3 & 0 & 0 & 0 & 0 \\
49 & 5 & 5 & 3 & 0 & 3 & 5 & 5 & 5 & 0 & 0 \\
\hline
\end{tabular}
\end{center}


\begin{thebibliography}{DW}

\bibitem[D]{davenport}
H.~Davenport.
\newblock On the class-number of binary cubic forms {I} and {II}.
\newblock {\em London Math. Soc.}, 26:183--198, 1951.
\newblock Corrigendum: ibid., 27:512, 1952.

\bibitem[DW]{dawra}
B.~Datskovsky and D.J. Wright.
\newblock The adelic zeta function associated with the space of binary cubic
  forms {II}: Local theory.
\newblock {\em J. Reine Angew. Math.}, 367:27--75, 1986.

\bibitem[IS]{is}T. Ibukiyama and H. Saito.
\newblock
On $L$-functions of ternary zero forms and exponential sums of Lee and Weintraub.
\newblock {\em J. Number Theory}, 48:252--257, 1994.

\bibitem[K]{kogiso}
T.~Kogiso.
\newblock Simple calculation of the residues of the adelic zeta function
  associated with the space of binary cubic forms.
\newblock {\em J. Number Theory}, 51:233--248, 1995.

\bibitem[N]{nakagawa}
J.~Nakagawa.
\newblock On the relations among the class numbers of binary cubic forms.
\newblock {\em Invent. Math.}, 134:101--138, 1998.


\bibitem[O]{ohno}
Y.~Ohno.
\newblock A conjecture on coincidence among the zeta functions associated with
  the space of binary cubic forms.
\newblock {\em Amer. J. Math.}, 119:1083--1094, 1997.

\bibitem[S1]{shintania}
T.~Shintani.
\newblock On {D}irichlet series whose coefficients are class-numbers of
  integral binary cubic forms.
\newblock {\em J. Math. Soc. Japan}, 24:132--188, 1972.

\bibitem[S2]{shintanib}
T.~Shintani.
\newblock On zeta-functions associated with vector spaces of quadratic forms.
\newblock {\em J. Fac. Sci. Univ. Tokyo, Sect IA}, 22:25--66, 1975.

\bibitem[T]{taniguchi}
T.~Taniguchi.
\newblock Distributions of discriminants of cubic algebras.
\newblock Preprint 2006, math.NT/0606109.

\bibitem[W]{wright}
D.J. Wright.
\newblock The adelic zeta function associated to the space of binary cubic
  forms part {I}: Global theory.
\newblock {\em Math. Ann.}, 270:503--534, 1985.

\end{thebibliography}

\end{document}